\documentclass[12pt]{amsart}
\usepackage{amsmath, amssymb, amsthm, enumerate,amsfonts,latexsym,mathrsfs,graphicx}
\usepackage{hyperref}

\usepackage{tikz}
\usetikzlibrary{decorations.markings}
\tikzstyle{vertex}=[circle, draw, inner sep=0pt, minimum size=4.5pt]
\newcommand{\vertex}{\node[vertex]}
\usepackage{calc}
\newcounter{Angle}

\newtheorem{theorem}{Theorem}[section]
\newtheorem{corollary}[theorem]{Corollary}
\newtheorem{lemma}[theorem]{Lemma}
\newtheorem{proposition}[theorem]{Proposition}

\newtheorem*{conjecture}{Conjecture}

\newtheorem*{theoremA}{Theorem~A}

\newcommand{\Pal}{P\'{a}lfy's Condition}

\newcommand{\Irr}{{\mathrm {Irr}}}
\newcommand{\cd}{{\mathrm {cd}}}
\newcommand{\Aut}{{\mathrm {Aut}}}

\newcommand{\Centralizer}{{\mathrm {C}}}

\newcommand{\PSL}{{\mathrm {PSL}}}
\newcommand{\PSp}{{\mathrm {PSp}}}
\newcommand{\PSU}{{\mathrm {PSU}}}
\newcommand{\SL}{{\mathrm {SL}}}

\newcommand{\GG}{{\mathscr{G}}}

\theoremstyle{definition}

\begin{document}

\title[Regular prime graphs]{Finite groups whose prime graphs are regular}

\author{Hung P. Tong-Viet}
\email{Tongviet@ukzn.ac.za}
\address{School of Mathematics, Statistics and Computer Science\\
University of KwaZulu-Natal\\
Pietermaritzburg 3209, South Africa}

\subjclass[2000]{Primary 20C15; 05C25}

\date{\today}

\keywords{character degree; prime graph; cubic graph; regular graph}

\begin{abstract}
Let $G$ be a finite group and let $\Irr(G)$ be the set of all  irreducible complex characters of
$G.$ Let $\cd(G)$ be the set of all character degrees of $G$ and denote by $\rho(G)$ the set of
primes which divide some character degrees of $G.$ The prime graph $\Delta(G)$ associated to $G$ is
a graph whose vertex set is $\rho(G)$ and there is an edge between two distinct primes $p$ and $q$
if and only if the product $pq$ divides some character degree of $G.$  In this paper, we show that
the prime graph $\Delta(G)$ of a finite group $G$ is $3$-regular if and only if it is a complete
graph with four vertices.
\end{abstract}

\maketitle

\section{Introduction}

Given a finite group $G,$ let $\Irr(G)$ be the set of all irreducible complex characters of $G$ and
let $\cd(G)=\{\chi(1)\:|\:\chi\in\Irr(G)\}$ be the set of character degrees of $G.$ The set of
primes which divide some character degrees of $G$ is denoted by $\rho(G).$ The {\em prime graph}
$\Delta(G)$ associated to $G$ is a graph whose vertex set is $\rho(G)$ and there is an edge between
two distinct primes $p$ and $q$ in $\rho(G)$ if and only if the product $pq$ divides some character
degree $a\in\cd(G).$ The prime graph $\Delta(G)$ of a finite group $G$ is a useful tool in studying
the character degree set $\cd(G).$ This graph has been studied extensively over the last $20$
years. We refer the readers to a recent survey by M. Lewis \cite{Lewis08} for results concerning
this graph and related topics.

In this paper, we are going to study the following question: Which graphs can occur as the prime
graphs of finite groups? This is one of the basic questions in the character theory of finite
groups. Although a complete answer to this question is yet to be found, many restrictions on the
structure of the prime graph $\Delta(G)$ have been obtained. For example, it is known that
$\Delta(G)$ has at most three connected components and if $\Delta(G)$ is connected, then its
diameter is bounded above by three. (See \cite[Theorems~6.4, 6.5]{Lewis08}). For finite solvable
groups,  \Pal~\cite{Palfy} asserts that given any three distinct primes in $\rho(G),$ there is
always an edge connecting two primes among those primes.  This condition is very useful in
determining which graphs can occur as the prime graphs of finite solvable groups. In particular,
this condition implies that if $G$ is finite solvable, then $\Delta(G)$ has at most two connected
components and if $\Delta(G)$ has exactly two connected components, then each component is
complete. Unfortunately, this condition does not hold true in general. Nevertheless, it was proved
in \cite{Moreto-Tiep} that if $\pi\subseteq \rho(G)$ with $|\pi|\geq 4,$ then there is an edge
connecting two distinct primes in $\pi.$

The main purpose of this paper is to classify all $k$-regular graphs which can occur as $\Delta(G)$
for some finite group $G,$ for $0\leq k\leq 3.$ Recall that a graph $\GG$ is called $k$-regular for
some integer $k\geq 0,$ if every vertex of $\GG$ has the same degree $k.$ Combining results in
\cite{Lewis-White13, Hung}, one can easily classify all $k$-regular prime graphs for $0\leq k\leq
2.$ (See Proposition~\ref{regular graphs of small valency} in Section~\ref{sec2}). In particular,
if $\Delta(G)$ is $2$-regular, then $\Delta(G)$ is a triangle or a square. For $3$-regular graphs,
we obtain the following result.

\begin{theoremA}
The prime graph $\Delta(G)$ of a finite group $G$ is $3$-regular if and only if it is a complete
graph with four vertices.
\end{theoremA}

Obviously, if $\Delta(G)$ is a complete graph with four vertices, then it is $3$-regular.
Therefore, we mainly  focus on the `only if' part. There are examples of both solvable and
nonsolvable groups whose prime graphs are complete graphs with four vertices. For nonsolvable
groups, we can simply take $G\cong {\rm A}_7,$ the alternating group of degree $7.$ For solvable
groups, we can take $G$ to be a direct product of two solvable groups $H$ and $K,$ where both
$\Delta(H)$  and $\Delta(K)$ are complete graphs with two vertices and $\rho(G)\cap \rho(H)$ is
empty.

We mention that an analogous result for conjugacy class sizes was obtained by Bianchi et al. in
\cite{Bianchi-Herzog}, where the authors proved that the common-divisor graph $\Gamma(G),$ defined
on the set of non-central conjugacy class sizes of a finite group $G$, is $3$-regular if and only
if it is a complete graph with four vertices and they conjectured that $\Gamma(G)$ is a $k$-regular
graph if and only if it is a complete graph with $k+1$ vertices. Recently, this conjecture has been proved in \cite{Bianchi}.

The paper is organized as follows. In Section~\ref{sec2}, we obtain an upper bound for the number
of vertices of the prime graph $\Delta(G)$ of a finite group $G$ in terms of the maximal degree $d$
and the independent number of $\Delta(G)$ under the assumption that $\Delta(G)$ contains no
subgraph isomorphic to a complete graph with $d+1$ vertices. (See Corollary~\ref{Upper bound}).
This result may be useful in studying the prime graphs with bounded degrees. In Section~\ref{sec3},
we prove Theorem~A for solvable groups. This is achieved in Theorem~\ref{3-regular solvable graph}.
Section~\ref{sec4} is devoted to proving Theorem~A for nonsolvable groups. This is the main part of
the paper. Finally, in the last section, for each even integer $k\geq 2,$ we construct a finite
solvable group whose prime graph is $k$-regular with $k+2$ vertices.

All groups in this paper are assumed to be finite, all characters are complex characters and all
graphs are finite, simple, undirected graphs (no loop nor multiple edge). We refer to \cite{Isaacs}
for the notation of character theory of finite groups and to \cite{Bollobas} for terminology in
graph theory. For an integer $n,$ we write $\pi(n)$ for the set of all prime divisors of $n.$ We
write $\pi(G)$ instead of $\pi(|G|)$ for the set of all prime divisors of $|G|.$  If $N\unlhd G$
and $\theta\in\Irr(N),$ then the inertia group of $\theta$ in $G$ is denoted by $I_G(\theta).$
Finally, we write $\Irr(G|\theta)$ for the set of all irreducible constituents of $\theta^G.$

\section{Prime graphs of groups}\label{sec2}
In this section, we recall some graph theoretic terminologies and some known results in both graph
and group theories which will be needed in this paper. We begin with some basic definitions and
results in graph theory.

Let $\mathscr{G}=(V,E)$ be a graph of order $n=|V|$ with vertex set $V$ and edge set $E.$ Let $v$
be a vertex of $\GG.$ The {\em degree} of $v$ is the number of edges of $\GG$ incident to $v.$ A
vertex $v\in V$ is said to be an {\em odd vertex} if its degree is odd. The following elementary
result, which is a consequence of the Hand-Shaking Lemma, is well known.

\begin{lemma}\label{hand shaking}  The number of odd vertices in a graph is even.
\end{lemma}
A graph $\GG$ is called {\em $k$-regular} (or \emph{regular of valency $k$}) for some integer
$k\geq 0,$ if every vertex of $\GG$ has the same degree $k.$ We call a $3$-regular graph a {\em
cubic graph}. From Lemma~\ref{hand shaking}, if $\GG$ is $k$-regular for some odd integer $k\geq
1,$ then the order of $\GG$ must be even since every vertex of $\GG$ is an odd vertex. For an
integer $n\geq 3,$ we denote by $K_n$ a complete graph of order $n,$ that is, a graph with $n$
vertices in which all pairs of distinct vertices are adjacent. A complete graph of order four is
called a {\em complete square}. A graph $\GG$ is said to be {\em $K_n$-free} for some $n\geq 3$ if
$\GG$ has no subgraph isomorphic to $K_n.$ Clearly, $\GG$ is $K_3$-free if and only if $\GG$ has no
triangle. Observe that a connected $k$-regular graph $\GG$ for some $k\geq 2$ is $K_{k+1}$-free if
and only if it is not $K_{k+1}.$ A subset $I$ of $V$ is an {\em independent set} if no two of its
elements are adjacent. The independent number $\alpha(\mathscr{G})$ of $\GG$ is the maximal size of
independent sets in $\mathscr{G}.$ Finally, the {\em chromatic number} of $\mathscr{G},$ denoted by
$\chi(\mathscr{G}),$ is the minimal number of colors in a vertex coloring of $\GG.$

It is well known that $\chi(\GG)\geq n/\alpha(\GG).$ (See for instance \cite[page~147]{Bollobas}).
Let $d$ be the maximal degree of a graph $\GG.$ From the definition, we have that $$\chi(\GG)\leq
d+1.$$ Brooks \cite{Brooks} classified all graphs for which the equality $\chi(\GG)=d+1$ holds. In
particular, if $\chi(\GG)=d+1,$ then $\GG$ contains $K_{d+1}$ or $d=2$ and $\GG$ contains an odd
cycle. Using this result, we obtain the following upper bound for the order of a graph in terms of
the independent number $\alpha(\GG)$ and the maximal degree $d$ of $\GG.$

\begin{lemma}\label{Brooks}
Let $\GG$ be a graph of order $n$ with maximal degree $d\geq 3.$ Suppose that $\GG$ is
$K_{d+1}$-free. Then $n\leq \alpha(\GG)d.$
 \end{lemma}
 \begin{proof}
By Brooks' Theorem \cite{Brooks}, we deduce that $\chi(\GG)\leq d.$ Since $\chi(\GG)\geq
n/\alpha(\GG),$ we obtain that $n/\alpha(\GG)\leq d$ or equivalently $n\leq \alpha(\GG)d$ as
wanted.
\end{proof}
Notice that if $\GG$ is connected with maximal degree $d\geq 3$ which is not $K_{d+1},$ then the
order of $\GG$ is bounded above by $\alpha(\GG)d.$

Let $G$ be a group and let $\pi\subseteq \rho(G).$ For solvable groups, P\'{a}lfy \cite{Palfy}
showed that there is always an edge between two primes in $\pi$ whenever $|\pi|\geq 3.$ For
arbitrary groups, Moret\'{o} and Tiep \cite{Moreto-Tiep} proved that a similar conclusion holds
provided that $|\pi|\geq 4.$ We summarise these results in the following lemma.

\begin{lemma} \label{Conditions} Let $G$ be a group and let $\pi\subseteq \rho(G).$
\begin{enumerate}
\item[$(1)$] \emph{(P\'{a}lfy's Condition~\cite[Theorem]{Palfy}).}
If $G$ is solvable and $|\pi|\geq 3,$ then there exist two distinct primes $u,v$ in $\pi$ and
$\chi\in\Irr(G)$ such that $uv\mid \chi(1).$

\item[$(2)$]\emph{(Moret\'{o}-Tiep's Condition~\cite[Main~Theorem]{Moreto-Tiep}).}
If $|\pi|\geq 4,$ then there exists $\chi\in\Irr(G)$ such that $\chi(1)$ is divisible by two
distinct primes in $\pi.$
\end{enumerate}
\end{lemma}

Translating these results into graph theoretic terminology, we obtain the following.

\begin{lemma}\label{Independent set}
Let $G$ be a group with prime graph $\Delta(G).$ Then $\alpha(\Delta(G))\leq 3.$ Moreover, if $G$
is solvable, then $\alpha(\Delta(G))\leq 2.$
\end{lemma}

Combining the previous two lemmas, we obtain an upper bound for the order of the prime graph
$\Delta(G)$ of a group $G.$

\begin{corollary}\label{Upper bound}
Let $G$ be a group with prime graph $\Delta(G).$ Suppose that the maximal degree of $\Delta(G)$ is
$d\geq 3$ and $\Delta(G)$ is $K_{d+1}$-free. Then $|\rho(G)|\leq 3d$  and if $G$ is solvable, then
$|\rho(G)|\leq 2d.$ In particular, if $\Delta(G)$ is connected $k$-regular for some $k\geq 3$ which
is not $K_{k+1},$ then $|\rho(G)|\leq 3k;$ and if $G$ is solvable, then $|\rho(G)|\leq 2k.$
\end{corollary}

We now classify regular graphs with small valency which might occur as $\Delta(G)$ for some group
$G$ using results we have collected so far. We first consider the case $\Delta(G)$ is disconnected.

\begin{lemma}\label{disconnected regular graphs}
Let $G$ be a group and let $k\geq 0$ be an integer. Suppose that $\Delta(G)$ is a disconnected
$k$-regular graph. Then $k=0.$
\end{lemma}

\begin{proof}
Assume first that $G$ is solvable. By \cite[Corollay~4.2]{Lewis08}, we know that $\Delta(G)$ has
exactly two connected components with vertex sets $\rho_1$ and $\rho_2,$ where $n_1=|\rho_1|\leq
n_2=|\rho_2|$ and that each connected component is complete. It follows that each vertex in
$\rho_1$ has degree $n_1-1$ and each vertex in $\rho_2$ has degree $n_2-1,$ respectively. Hence,
$n_1-1=k=n_2-1,$ so $n_1=n_2.$ Now \cite[Theorem~4.3]{Lewis08} yields that $n_2\geq 2^{n_1}-1,$
which forces $n_1=n_2=1$ and so $k=0.$

Assume now that $G$ is nonsolvable. By \cite[Theorem~6.4]{Lewis08}, $\Delta(G)$ has at most $3$
connected components and one of which is an isolated vertex. Since $\Delta(G)$ is $k$-regular, we
deduce that $k=0.$ The proof is now complete.
\end{proof}

The next result gives a classification of all $k$-regular graphs with $0\leq k\leq 2,$ which can
occur as $\Delta(G)$ for some group $G.$

\begin{proposition}\label{regular graphs of small valency} Let $G$ be a group. Suppose that the
prime graph $\Delta(G)$ is $k$-regular for some $k$ with $0\leq k\leq 2.$ Then the following hold.
\begin{enumerate}
  \item[$(1)$] If $k=0,$ then $\Delta(G)$ has at most $3$ vertices and each vertex of $\Delta(G)$ is isolated.
  \item[$(2)$] If $k=1,$ then $\Delta(G)$ is isomorphic to $K_2,$ a complete graph with $2$
  vertices. In particular, $G$ is solvable.
  \item [$(3)$] If $k=2,$ then $\Delta(G)$ is either a triangle or a square. Moreover, $G$ is
  solvable if $\Delta(G)$ is a square.
\end{enumerate}
\end{proposition}
\begin{proof}
If $k=0,$ then the result is clear since $\Delta(G)$ has at most three connected components. Assume
now that $k\geq 1.$ By Lemma~\ref{disconnected regular graphs}, $\Delta(G)$ is a connected
$k$-regular graph. If $k=1,$ then $\Delta(G)$ is a line with two vertices and by using
Ito-Michler's Theorem \cite[Theorem~5.5]{Michler} and Burnside's $p^aq^b$ Theorem
\cite[Theorem~3.10]{Isaacs}, we deduce that $G$ is solvable. Finally, assume that $k=2.$ It follows
that $\Delta(G)$ is a cycle of length $n=|\rho(G)|\geq 3.$ If $n=3,$ then $\Delta(G)$ is a triangle
and $G$ could be solvable or nonsolvable. Assume now that $n>3.$ It follows from
\cite[Theorem~C]{Hung} that $n=4$ and so $\Delta(G)$ is a square. Now
\cite[Theorem~B]{Lewis-White13} yields that $G$ is solvable.
\end{proof}

In the last result of this section, we eliminate all but four cubic graphs of order at least $6$
which might occur as the prime graph of some group. These graphs will be ruled out in the next two
sections.

\begin{proposition}\label{classification of 3-regular graphs}
If the prime graph $\Delta(G)$ of a group $G$  is a connected $3$-regular graph with $|\rho(G)|\geq
6,$ then $\Delta(G)$ is isomorphic to one of the graphs in Figures~\ref{Fig1}-\ref{Fig4}.
\end{proposition}

\begin{figure}[ht]
\begin{center}

\[\begin{tikzpicture}
    \vertex[fill] (p1) at (0,1)[label=left:$p_1$] {};
    \vertex[fill] (p3) at (1,0)  [label=below:$p_3$]  {};
    \vertex[fill] (p2) at (1,2)  [label=above:$p_2$]  {};
    \vertex[fill] (p5) at (3,2)  [label=above:$p_5$]  {};
    \vertex[fill] (p6) at (3,0)  [label=below:$p_6$]  {};
    \vertex[fill] (p4) at (4,1)  [label=right:$p_4$]  {};

    \path
        (p1) edge (p2)
        (p1) edge (p3)
        (p1) edge (p4)
        (p2) edge (p3)
        (p2) edge (p5)
        (p4) edge (p5)
        (p4) edge (p6)
        (p5) edge (p6)
        (p3) edge (p6)

        ;

\end{tikzpicture}\]
\end{center}
\caption{A cubic graph of order six} \label{Fig1}
\end{figure}
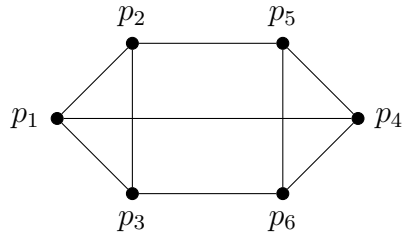

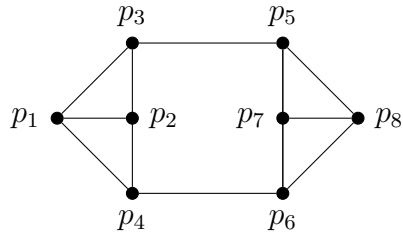
\begin{figure}[ht]
\begin{center}

\[\begin{tikzpicture}
    \vertex[fill] (p1) at (0,1)[label=left:$p_1$] {};
    \vertex[fill] (p2) at (1,1)  [label=right:$p_2$]  {};
    \vertex[fill] (p3) at (1,2)  [label=above:$p_3$]  {};
    \vertex[fill] (p4) at (1,0)  [label=below:$p_4$]  {};
    \vertex[fill] (p5) at (3,2)  [label=above:$p_5$]  {};
    \vertex[fill] (p6) at (3,0)  [label=below:$p_6$]  {};
    \vertex[fill] (p7) at (3,1)  [label=left:$p_7$]  {};
    \vertex[fill] (p8) at (4,1)  [label=right:$p_8$]  {};

    \path
        (p1) edge (p3)
        (p1) edge (p4)
        (p1) edge (p2)
        (p3) edge (p4)
        (p6) edge (p5)
        (p8) edge (p5)
        (p8) edge (p6)
        (p5) edge (p6)
        (p8) edge (p7)
        (p3) edge (p5)
        (p4) edge (p6)

        ;

\end{tikzpicture}\]
\end{center}
\caption{A cubic graph of order eight with four triangles} \label{Fig2}
\end{figure}

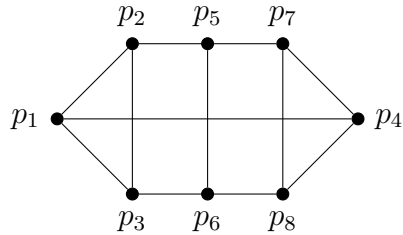
\begin{figure}[ht]
\begin{center}

\[\begin{tikzpicture}
    \vertex[fill] (p1) at (0,1)[label=left:$p_1$] {};
    \vertex[fill] (p2) at (1,2)  [label=above:$p_2$]  {};
    \vertex[fill] (p3) at (1,0)  [label=below:$p_3$]  {};
    \vertex[fill] (p4) at (4,1)  [label=right:$p_4$]  {};
    \vertex[fill] (p5) at (2,2)  [label=above:$p_5$]  {};
    \vertex[fill] (p6) at (2,0)  [label=below:$p_6$]  {};
    \vertex[fill] (p7) at (3,2)  [label=above:$p_7$]  {};
    \vertex[fill] (p8) at (3,0)  [label=below:$p_8$]  {};

    \path
        (p1) edge (p4)
        (p2) edge (p7)
        (p3) edge (p8)
        (p2) edge (p3)
        (p6) edge (p5)
        (p8) edge (p7)
        (p1) edge (p2)
        (p1) edge (p3)
        (p4) edge (p7)
        (p4) edge (p8)

        ;

\end{tikzpicture}\]
\end{center}
\caption{A cubic graph of order eight with two triangles} \label{Fig3}
\end{figure}

\begin{figure}[ht]
\begin{center}

\[\begin{tikzpicture}
    \vertex[fill] (p1) at (0,1)[label=left:$p_1$] {};
    \vertex[fill] (p2) at (1,2)  [label=above:$p_2$]  {};
    \vertex[fill] (p3) at (1,0)  [label=below:$p_3$]  {};
    \vertex[fill] (p4) at (4,1)  [label=right:$p_4$]  {};
    \vertex[fill] (p5) at (2,2)  [label=above:$p_5$]  {};
    \vertex[fill] (p6) at (2,0)  [label=below:$p_6$]  {};
    \vertex[fill] (p7) at (3,2)  [label=above:$p_7$]  {};
    \vertex[fill] (p8) at (3,0)  [label=below:$p_8$]  {};

    \path
        (p1) edge (p4)
        (p2) edge (p7)
        (p3) edge (p8)
        (p2) edge (p3)
        (p6) edge (p7)
        (p8) edge (p5)
        (p1) edge (p2)
        (p1) edge (p3)
        (p4) edge (p7)
        (p4) edge (p8)
        ;

\end{tikzpicture}\]
\end{center}
\caption{A cubic graph of order eight with one triangle} \label{Fig4}
\end{figure}
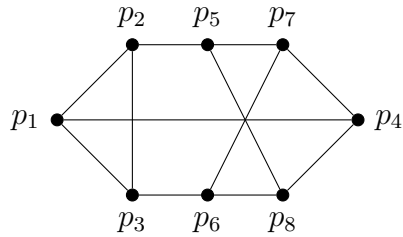

\begin{proof}
Suppose that $\Delta(G)$ is a connected $3$-regular graph with $|\rho(G)|\geq 6$ for some group
$G.$ Since $\Delta(G)$ is $3$-regular, every vertex of $\Delta(G)$ is an odd vertex and thus by
Lemma~\ref{hand shaking}, $|\rho(G)|$ must be even. By Corollary~\ref{Upper bound}, $|\rho(G)|\leq
9$ since it is connected but it is not isomorphic to $K_4.$ Therefore, $|\rho(G)|=6$ or $8.$
Writing $\rho(G)=\{p_i\}_{i=1}^n$ with $n=|\rho(G)|.$ By \cite[Theorem~A]{Hung}, we know that
$\Delta(G)$ always contains a triangle with vertex set, say $\{p_1,p_2,p_3\}.$ We consider the
following cases.

\medskip {\bf Case 1.} $\Delta(G)$ has a vertex which is adjacent to two distinct vertices in
$\{p_i\}_{i=1}^3.$

Without loss of generality, assume that $p_4\in\rho(G)$ is adjacent to two distinct vertices in
$\{p_i\}_{i=1}^3,$ say $p_1$ and $p_2.$ Since $p_1$ and $p_2$ are joined to each other and to both
$p_3$ and $p_4,$ they are not adjacent to any other vertices. As both $p_3$ and $p_4$ have degree
$3$ in $\Delta(G),$ one of the following cases holds.

\smallskip
$(\textrm{i})$ Both $p_3$ and $p_4$ are adjacent to the same vertex, say $p_5.$

Assume that $|\rho(G)|=6.$ As $\Delta(G)$ is connected, $p_6$ must be adjacent to $p_5$ but it
cannot be adjacent to any other vertices since otherwise $\Delta(G)$ would have a vertex of degree
at least four, which is impossible. However, this case cannot happen since we cannot add more edges
to this graph to obtain a cubic graph of order six.

Now assume that $|\rho(G)|=8.$ Let $\tau=\{p_6,p_7,p_8\}.$ As both $p_3$ and $p_4$ are adjacent to
$p_1,p_2$ and $p_5,$ they are not joined to each other and to any vertices in $\tau.$ So, for any
$u\neq v\in\tau,$ by applying Moret\'{o}-Tiep's Condition for $\{u,v,p_3,p_4\},$ we see that $u$
and $v$ must be joined to each other. In particular, $\tau$ is a vertex set of a triangle in
$\Delta(G).$  As $\Delta(G)$ is connected, $p_5$ must be adjacent to some vertex in $\tau,$ say
$p_6.$ But then we cannot add more edges to this graph to obtain a cubic graph since $p_7$ and
$p_8$ cannot be joined to any other vertices. Hence, this case cannot happen.

\smallskip
$(\textrm{ii})$ $p_3$ and $p_4$ are adjacent to two distinct vertices, say $p_5$ and $p_6,$
respectively.

If $|\rho(G)|=6,$ then we cannot add more edges to this graph to obtain a cubic graph. Hence, this
case cannot occur.

Therefore, we can assume that $|\rho(G)|=8.$  Let $\tau=\{p_7,p_8\}.$ Notice that $p_3$ and $p_4$
are not adjacent in $\Delta(G).$ By Moret\'{o}-Tiep's Condition for $\{p_3,p_4,p_7,p_8\},$ $p_7$
and $p_8$ are joined to each other. Assume that $p_5$ and $p_6$ are adjacent.  Since $\Delta(G)$ is
connected and each $p_i,1\leq i\leq 4,$ cannot be adjacent to any other vertices, $p_5$ must be
adjacent to either $p_7$ or $p_8,$ say $p_7.$ As the degree of $p_6$ is three, it must be adjacent
to either $p_7$ or $p_8.$ However, both cases are impossible as we cannot add more edges to this
graph to obtain a cubic graph. Thus, we can assume that $p_5$ and $p_6$ are not joined to each
other. Hence, each of them must be adjacent to both $p_7$ and $p_8$ and thus $\Delta(G)$ is
isomorphic to the graph in Fig.~\ref{Fig2}.

\medskip {\bf Case 2.} No vertex in $\rho(G)$ is adjacent to two distinct vertices in
$\{p_i\}_{i=1}^3.$

As each $p_i,i=1,2,3,$ has degree three in $\Delta(G),$ there exist $\{p_j\}_{j=4}^6\subseteq
\rho(G)-\{p_i\}_{i=1}^3$ such that each $p_k$ is adjacent to $p_{k+3}$ for $k=1,2,3.$

Assume first that $|\rho(G)|=6.$ Then $\rho(G)=\{p_i\}_{i=1}^6.$ Clearly, $\Delta(G)$ is a
$3$-regular graph   if and only if  $\{p_i\}_{i=4}^6$ forms a triangle. Hence, $\Delta(G)$ is the
graph in Fig.~\ref{Fig1}. Assume next that $|\rho(G)|=8.$ Let $\tau=\{p_7,p_8\}.$ Then the
following cases hold.

\smallskip
$(\textrm{i})$ Assume that some vertex $p_j$ with $4\leq j\leq 6$ is adjacent to the remaining
vertices in $\{p_i\}_{i=4}^6.$ Without loss of generality, assume that $p_6$ is adjacent to $p_4$
and $p_5.$ Since $\Delta(G)$ is connected and $|\rho(G)|=8,$ $p_4$ and $p_5$ are not adjacent in
$\Delta(G).$ By Moret\'{o}-Tiep's Condition for $\{p_1,p_6,p_7,p_8\},$  $p_7$ and $p_8$ are
adjacent in $\Delta(G).$ As $p_4$ has degree three, it must be adjacent to $p_7$ or $p_8,$ say
$p_7.$ Similarly, $p_5$ is adjacent to either $p_7$ or $p_8.$ However, both cases are impossible as
we cannot add more edges to this graph to obtain a cubic graph.

\smallskip
$(\textrm{ii})$ Assume next that there is exactly one edge among vertices $\{p_i\}_{i=4}^6,$ say
$\{p_5,p_6\}.$ Then $p_4$ is not adjacent to $p_5,p_6.$ Since $p_4$ has degree $3$ in $\Delta(G)$
and it is not adjacent to any vertices in $\{p_i\}_{i=2}^6,$ it is adjacent to both $p_7$ and
$p_8.$ Similarly, as $p_5$ has degree $3$ in $\Delta(G),$ it is adjacent to either $p_7$ or $p_8,$
say $p_7.$ It follows that $p_6$ is adjacent to either $p_7$ or $p_8.$ If $p_6$ is adjacent to
$p_7,$ then $p_8$ is not adjacent to any other vertices, except for $p_4,$ which is impossible.
Thus $p_6$ is adjacent to $p_8.$ By joining $p_7$ and $p_8,$ we obtain the graph in
Fig.~\ref{Fig3}.

\smallskip
$(\textrm{iii})$ Assume that there is no edge among vertices $p_j,4\leq j\leq 6.$ For each $i\in
\{1,2,3\},$ $p_i$ is joined to $p_{i+3}$ and to the remaining vertices in $\{p_\ell\}_{\ell=1}^3,$
so it is not adjacent to any other vertices. Hence, each vertex $p_j, 4\leq j\leq 6,$ is adjacent
to both vertices $p_7$ and $p_8.$ Thus, we obtain the graph in Fig.~\ref{Fig4}. The proof is now
complete.
\end{proof}

\section{Prime graphs of solvable groups}\label{sec3}
In this section, we prove Theorem~A for the solvable groups. We first eliminate the graph in
Fig.~\ref{Fig1} from being the prime graph of some solvable group.

\begin{lemma}\label{cubical graph with 6 vertices} Suppose that the prime graph
$\Delta(G)$ of a group $G$ is isomorphic to the graph in Fig.~\ref{Fig1}. Then $G'=G''.$ In
particular, $G$ is nonsolvable.
\end{lemma}
\begin{proof} By way of contradiction, assume that $G'\neq G''.$ Then $G/G''$ is a nonabelian
solvable group. Let $N$ be a maximal normal subgroup of $G$ such that $G/N$ is a minimal nonabelian
solvable group.  By \cite[Lemma~12.3]{Isaacs}, the following cases hold.

\medskip
{\bf Case 1.} $G/N$ is a nonabelian $r$-group for some prime $r.$

In this case, $G/N$ has an irreducible character $\tau\in\Irr(G/N)$ of degree $r^a$ for some prime
$r$ and some integer $a\geq 1.$ We now claim that $r$ is adjacent to every prime in
$\rho(G)-\{r\}.$ Indeed, for every prime $p\in\rho(G)-\{r\},$ there exists $\chi\in\Irr(G)$ with
$p\mid \chi(1).$ If $r\mid \chi(1),$ then $r$ and $p$ are joined to each other and we are done. So,
we can assume that $r\nmid \chi(1).$ It follows that $\gcd(\chi(1),|G:N|)=1,$ hence
$\chi_N\in\Irr(G).$ By Gallagher's Theorem \cite[Corollary~6.17]{Isaacs}, $\chi\tau\in\Irr(G),$ so
$pr\mid \chi(1)\tau(1),$ therefore, $p$ and $r$ are adjacent in $\Delta(G).$ Thus, $r$ is adjacent
to every prime in $\rho(G)-\{r\}.$ Since $|\rho(G)|=6,$ $r$ has degree five in $\Delta(G),$ which
is a contradiction.

\medskip
{\bf Case 2.} $G/N$ is a Frobenius group with Frobenius kernel $F/N,$ an elementary abelian
$r$-group for some prime $r,$ and $f=|G:F|\in\cd(G)$ with $\gcd(r,f)=1.$

By \cite[Theorem~12.4]{Isaacs}, we know that for every $\psi\in\Irr(F),$ either $f\psi(1)\in\cd(G)$
or $r\mid \psi(1).$ Since $\Delta(G)$ has no complete subgraph isomorphic to $K_4,$ we deduce that
$|\pi(\chi(1))|\leq 3$ for all $\chi\in\Irr(G).$ Notice that $\Delta(G)$ has exactly two triangles
with vertex sets $\{p_1,p_2,p_3\}$ and $\{p_4,p_5,p_6\}.$ The remaining edges of $\Delta(G)$ are
$\{p_i,p_{i+3}\}$ for $1\leq i\leq 3.$ As $f\in\cd(G),$ we have that $|\pi(f)|\leq 3.$ Hence,
$|\pi(f)\cup\{r\}|\leq 4,$ so there exists $j\in\{1,2,3\}$ such that $r\not\in\{p_j,p_{j+3}\}$ and
$\pi(f)\not\subseteq \{p_j,p_{j+3}\}.$ Thus, we can find $s\in\pi(f)$ with $s\not\in
\{p_j,p_{j+3}\}.$ Let $\chi\in\Irr(G)$ such that $p_jp_{j+3}\mid \chi(1).$ As $\Delta(G)$ contains
only two triangles with vertex sets $\{p_1,p_2,p_3\}$ and $\{p_4,p_5,p_6\},$ respectively, we
deduce that $\pi(\chi(1))=\{p_j,p_{j+3}\}.$ Let $\theta\in\Irr(F)$ be an irreducible constituent of
$\chi_F.$ Since $\theta(1)\mid \chi(1),$ we deduce that $r\nmid \theta(1),$ hence
$f\theta(1)\in\cd(G).$ Writing $\chi(1)=k\theta(1).$ Then $k\mid \gcd(\chi(1),f),$ so $s\nmid k,$
hence $s\mid f/k.$ We now have that $f\theta(1)=f\chi(1)/k=\chi(1)(f/k)$ is divisible by $s,p_j$
and $p_{j+3},$ so the subgraph on $\{s,p_j,p_{j+3}\}$ of $\Delta(G)$ is a triangle, a
contradiction. Therefore, $G'=G''$ as wanted.

Finally, if $G$ is solvable, then $G'=G''=1$ so  $G$ is abelian, which is impossible since
$|\rho(G)|=6.$ The proof is now complete.
\end{proof}

We now prove the main result of this section.
\begin{theorem}\label{3-regular solvable graph} Let $G$ be a solvable group. If $\Delta(G)$ is a
cubic graph, then $\Delta(G)$ is isomorphic to a complete graph  of order four.
\end{theorem}

\begin{proof}
Suppose that $\Delta(G)$ is a cubic graph for some solvable group $G.$ Then $|\rho(G)|\geq 4$ and
every vertex of $\Delta(G)$ has the same degree $3.$ By Lemma~\ref{disconnected regular graphs},
$\Delta(G)$ is connected. As every vertex of $\Delta(G)$ is an odd vertex, Lemma~\ref{hand shaking}
implies that  $|\rho(G)|$ is even. If $|\rho(G)|=4,$ then $\Delta(G)$ is a complete square and we
are done. So, assume that $|\rho(G)|\geq 6.$ By Corollary~\ref{Upper bound}, we have that
$|\rho(G)|\leq 6,$ so $|\rho(G)|=6.$ By Proposition~\ref{classification of 3-regular graphs},
$\Delta(G)$ is isomorphic to the graph in Fig.~\ref{Fig1}. Now Lemma~\ref{cubical graph with 6
vertices} yields a contradiction. Thus $\Delta(G)$ must be a complete square.
\end{proof}

\section{Prime graphs of nonsolvable groups}\label{sec4}

In this section, we give a proof of Theorem~A for nonsolvable groups. By the Ito-Michler theorem
\cite[Theorem~5.5]{Michler}, we know that $\rho(G)=\pi(G)$ for any almost simple groups $G$ since
$G$ has no nontrivial normal abelian Sylow subgroup. This fact will be used without further
reference. We first classify all simple groups whose prime graphs are $K_4$-free.
\begin{lemma}\label{simple groups} Let $S$ be a nonabelian simple group. Suppose that the prime
graph $\Delta(S)$ is $K_4$-free. Then one of the following cases holds.
\begin{enumerate}
  \item[$(1)$] $S\cong {\rm M}_{11}$ or ${\rm J}_1;$
  \item[$(2)$] $S\cong {\rm A}_n$ with $n\in\{5,6,8\};$
  \item[$(3)$] $S\cong \PSL_2(q)$ with $q=p^f\geq 4$ and $|\pi(q\pm 1)|\leq 3,$ where $p$ is prime;
  \item[$(4)$] $S\cong\PSL_3(q)$ with $q\in\{3,4,8\};$
  \item[$(5)$] $S\cong\PSU_3(q)$ with $q\in\{3,4,9\};$
  \item[$(6)$] $S\cong \PSp_4(3)$ or ${^2{\rm B}_2(q^2)}$ with $q^2=2^3$ or $2^5.$
\end{enumerate}
\end{lemma}

\begin{proof}

If $|\pi(S)|=3,$ then $S$ is isomorphic to one of the following simple groups
\[\textrm{A}_5,\textrm{A}_6,\PSp_4(3)\cong\PSU_4(2),\PSL_2(7),\PSL_2(8),\PSU_3(3),\PSL_3(3),\PSL_2(17)\]
 by \cite[Table~1]{Huppert-Lempken}. Clearly, these groups are $K_4$-free and they appear somewhere in the conclusion of the lemma.
Hence, we can assume that $|\pi(S)|\geq 4.$

If $S$ is a sporadic simple group or an alternating group, then by using Theorems~$2.1$ and $3.1$
in \cite{White}, we can easily deduce that $\Delta(S)$ is $K_4$-free if and only if $S$ is one of
the groups in  $(1)$ and $(2)$.

Assume that $S\cong{}^2\textrm{B}_2(q^2)$ with $q^2=2^{2m+1}$ and $m\geq 1.$ By
\cite[Theorem~4.1]{White}, the subgraph of $\Delta(S)$ on $\pi((q^2-1)(q^4+1))$ is complete. Since
$\Delta(S)$ is $K_4$-free,  $|\pi((q^2-1)(q^4+1))|\leq 3$ and thus $3\leq |\pi(S)|\leq 4$ as
$\pi(S)=\{2\}\cup \pi((q^2-1)(q^4+1)).$ Hence, $|\pi(S)|=4$ and so by
\cite[Table~2]{Huppert-Lempken}, $m=1$ or $2.$ These cases appear in $(6).$

If $S\cong\PSL_2(q),$ then the result is clear as the character degree set of $S$ is known. In
particular, as $q\pm 1\in\cd(G),$  $|\pi(q\pm 1)|\leq 3.$ This gives $(3).$

Assume that $S\cong \PSL_3(q)$ or $\PSU_3(q)$ $(q>2).$ If $S\cong \PSL_3(4)$ or $\PSL_4(2)\cong
\textrm{A}_8,$ then $S$ is $K_4$-free. Assume next that $S$ is not one of these groups. By
Theorems~$5.3$ and $5.5$ in \cite{White}, the subgraph of $\Delta(S)$ on $\pi(S)-\{p\}$ is
complete, where $q=p^f$ and $p$ is a prime.  Obviously, $|\pi(S)|=4$ and thus $(4)$ and $(5)$ hold
by using \cite[Table~2]{Huppert-Lempken}.

For the remaining simple groups, $\Delta(S)$ is complete by using \cite{White} again and thus
$\Delta(S)$ contains a subgraph isomorphic to $K_4$ since $|\pi(S)|\geq 4.$ Hence, $\Delta(S)$ is
not $K_4$-free in any of these cases. This completes the proof.
\end{proof}

We now deduce the following result for almost simple groups.
\begin{lemma}\label{almost simple groups} Let $S$ be a nonabelian simple group and let $G$ be an
almost simple group with $S\unlhd G\leq \Aut(S).$ Suppose that $\Delta(G)$ is $K_4$-free with
maximal degree $d\leq 3.$ Then $S$ is one of the following simple groups:
\begin{enumerate}
  \item[$(1)$] $S\cong {\rm M}_{11}$ and $|\pi(S)|=4;$
  \item[$(2)$] $S\cong {\rm A}_n$ with $n\in\{5,6,8\}$ and $|\pi(S)|\leq 4;$
  \item[$(3)$] $S\cong \PSL_2(q)$ with $q=p^f\geq 4$ and $|\pi(q\pm 1)|\leq 3,$ where $p$ is prime;
  Moreover, if $q$ is odd, then $|\pi(q^2-1)|\leq 4$ so $|\pi(S)|\leq 5;$ and if $q$ is even, then $|\pi(S)|\leq 7;$
  \item[$(4)$] $S\cong\PSL_3(q)$ with $q\in\{2,3,4,8\}$ and $|\pi(S)|\leq 4;$
  \item[$(5)$] $S\cong\PSU_3(q)$ with $q\in\{3,4,9\}$ and $|\pi(S)|\leq 4;$
  \item[$(6)$] $S\cong \PSp_4(3)$ or ${^2{\rm B}_2(q^2)}$ with $q^2=2^3$ or $2^5$ and $\pi(S)\leq 4;$
\end{enumerate} and one of the following cases holds.
\begin{enumerate}
  \item[$(\textrm{i})$] $\pi(G)=\pi(S)$ and either $3\leq |\pi(S)|\leq 4$ or $S\cong\PSL_2(q)$ with
$|\pi(q^2-1)|\geq 4,$ $q$ being a prime power and $|\pi(S)|\leq 7$  or
  \item[$(\textrm{ii})$]
$\pi(G)=\pi(S)\cup \{r\},$ $S\cong\PSL_2(q),$ $q$ being a prime power, $r\in\pi(G)-\pi(S)$ is
adjacent to every prime in $\pi(q^2-1),$ and $|\pi(G)|=|\pi(S)|+1=5.$
\end{enumerate}
In all cases, we have $|\pi(G)|\leq 7.$ Moreover, if $|\pi(G)|\geq 6,$ then $G=S,$ where
$S\cong\PSL_2(2^f)$ with $f\geq 10;$ and if $|\pi(G)|\geq 5,$ then $S\cong\PSL_2(q)$ for some prime
power $q\geq 11.$
\end{lemma}
\begin{proof}
As $S\unlhd G,$ the graph $\Delta(S)$ is a subgraph of $\Delta(G)$ so $\Delta(S)$ is $K_4$-free and
all vertices of $\Delta(S)$ have degree at most $3.$ Now the possibilities for $S$ are obtained
from Lemma~\ref{simple groups} by excluding simple groups having a vertex with degree exceeding
$3.$ Now $(i)$ is obvious. For $(ii),$ let $\tau=\pi(G)-\pi(S)$ and assume that $\tau$ is nonempty.
Clearly, this implies that either $S\cong{}^2\textrm{B}_2(8)$ or $S\cong\PSL_2(q)$ for some prime
power $q\geq 4.$ If the first case holds, then $G={}^2\rm{B}_2(8)\cdot 3;$ however $\Delta(G)$ is
not $K_4$-free by using \cite{ATLAS}. So, assume that the latter case holds. Then $|\pi(S)|\geq 4$
and hence $|\pi(q^2-1)|\geq 3.$ Observe that all primes in $\tau$ are odd and if $m$ is the product
of all distinct primes in $\tau,$ then by \cite[Theorem~A]{White13}, we deduce that $m(q\pm 1)$
divide some character degrees of $G.$ As $|\pi(S)|\geq 4,$ we have that $|\pi(q+\delta)|\geq 2$ for
some $\delta\in\{\pm 1\}$ so that $|\pi(m)|=1$ since otherwise $m(q+\delta)$ would have four
distinct prime divisors, a contradiction. Thus $\pi(m)=\pi(G)-\pi(S)=\{r\}$ and $r$ is adjacent to
all primes in $\pi(q^2-1)$ as wanted. Furthermore, as $r$ has degree at most $3$ in $\Delta(G),$ we
have that $|\pi(q^2-1)|=3,$ hence $|\pi(S)|=4$ and $|\pi(G)|=5.$

Clearly, $|\pi(G)|\leq 7$ by $(i)$ and $(ii).$ Now assume that $|\pi(G)|\geq 6.$ By $(ii),$ we have
that $\pi(G)=\pi(S).$ In particular, $|\pi(S)|\geq 6.$ Hence, $S\cong\PSL_2(2^f)$ with $f\geq 10.$
We next claim that $G=S.$ Suppose by contradiction that $G\neq S.$ By \cite[Theorem~A]{White13}
again, we know that $|G:S|(2^f\pm 1)\in \cd(G).$ Let $u$ be a prime divisor of $|G:S|.$ Then $u$ is
adjacent to every prime in $\pi(2^{2f}-1)-\{u\}.$ As $|\pi(S)|\geq 6,$ we have that
$|\pi(2^{2f}-1)|\geq 5$ and thus the degree of $u$ in $\Delta(G)$ is at least four, contradicting
the hypothesis of the lemma. The last claim is clear. The proof is now complete.
\end{proof}

We will need several results before we can prove the main result of this section. In the next two
lemmas, we obtain some restrictions on the structure of  nonsolvable groups whose prime graph is a
subgraph of a $3$-regular graph  with at least five vertices.
 \begin{lemma}\label{characteristic simple}
Suppose that the prime graph $\Delta(G)$ of some nonsolvable group $G$ is $K_4$-free with maximal
degree $d\leq 3$ and $|\rho(G)|\geq 5.$ Then every nonabelian chief factor of $G$ is simple.
 \end{lemma}

\begin{proof}
Assume that $M/N$ is a nonabelian chief factor of $G.$ Then $M/N\cong S^k$ for some nonabelian
simple group $S$ and some integer $k\geq 1.$  Let $C/N=\Centralizer_{G/N}(M/N).$ Then $C\unlhd G$
and $G/C$ has a unique minimal normal subgroup $M/N.$ Assume that $k>1.$ Since $G/C$ has no
nontrivial normal abelian Sylow subgroup, Ito-Michler's theorem yields that $\rho(G/C)=\pi(G/C)$
and thus $|\rho(G/C)|=|\pi(G/C)|\geq 3$ by \cite[Theorem~3.10]{Isaacs}. Applying \cite[Main
Theorem]{Lewis-McVey}, the prime graph $\Delta(G/C)$ is complete. Since $\Delta(G/C)$ is a subgraph
of $\Delta(G)$ which is $K_4$-free, $\Delta(G/C)$ is also $K_4$-free and thus $\Delta(G/C)$ must be
a triangle with vertex set $\{p_i\}_{i=1}^3.$  Let $L$ be a normal subgroup of $MC$ such that
$L/C\cong S.$ Since $\rho(G/C)=\pi(G/C),$ for any $r\in \pi:=\rho(G)-\pi(G/C),$ we deduce that
$r\in\rho(C)$ and thus there exists $\theta\in\Irr(C)$ with $r\mid \theta(1).$ By
\cite[Lemma~$4.2$]{Hung}, either $\theta$ extends to $\theta_0\in\Irr(L)$ or $\psi(1)/\theta(1)$ is
divisible by two distinct primes in $\pi(L/C)$ for some $\psi\in\Irr(L|\theta).$ If the first case
holds, then $r$ is adjacent to every prime $p_i,1\leq i\leq 3,$ and so the subgraph of $\Delta(G)$
on $\{r,p_1,p_2,p_3\}$ is a complete square, which is impossible. If the latter case holds, then
$r$ is adjacent to two distinct primes, say $p_i\neq p_j.$ As $|\rho(G)|\geq 5,$ we can find $r\neq
s\in \pi,$ and thus with the same argument as above, $s$ is also adjacent to two distinct primes in
$\{p_k\}_{k=1}^3.$ It follows that there exists a prime $p_m,$ $1\leq m\leq 3,$ such that $p_m$ is
adjacent to both $r$ and $s$ in $\Delta(L).$ As $\Delta(G/C)$ is a triangle, $p_m$ is adjacent to
every prime in $\pi(G/C)-\{p_m\},$ so its degree in $\Delta(G)$ is at least $4,$ a contradiction.
Therefore, $k=1$ as wanted.\end{proof}

\begin{lemma}\label{solvable radical} Assume the hypotheses of Lemma~\ref{characteristic simple}.
Let $N$ be the solvable radical of $G$ and let  $M/N$ be a chief factor of $G.$ Then $G/N$ is
almost simple with socle $M/N.$
\end{lemma}

\begin{proof} Since $N$ is the largest normal solvable subgroup of $G,$ $M/N$ is
nonabelian and so by Lemma~\ref{characteristic simple}, $M/N\cong S,$ where $S$ is a nonabelian
simple group. Let $C/N=\Centralizer_{G/N}(M/N).$ It suffices to show that $C=N.$ Suppose by
contradiction that $C\neq N$ and let $K/N$ be a minimal normal subgroup of $G/N$ with $K\leq C.$
Then $K/N$ is a nonabelian chief factor of $G,$ so by Lemma~\ref{characteristic simple} again,
$K/N\cong T,$ where $T$ is a nonabelian simple group. Notice that $K\cap M=N,$ $CM/N\cong C/N\times
M/N\unlhd G/N$ is a direct product and $\pi(T)\subseteq \pi(C/N),$ hence $2\in\rho(C/N)\cap
\pi(M/N).$ It follows that $2$ is adjacent to every prime in $\rho(C/N)\cup \pi(M/N).$ As the
degree of $2$ in $\Delta(G)$ is at most $3,$ we deduce that $3\leq |\rho(C/N)\cup \pi(M/N)|\leq 4.$

Assume first that $|\rho(C/N)\cup \pi(M/N)|= 4.$ Since $|\rho(C/N)|\geq |\pi(T)|\geq 3$ and
$|\pi(M/N)|\geq 3,$ we deduce that $|\rho(C/N)\cap \pi(M/N)|\geq 2.$ Writing $$\pi:=\rho(C/N)\cap
\pi(M/N)=\{r_i\}_{i=1}^k,$$ where $r_i,1\leq i\leq k,$ are distinct primes and $k\geq 2.$ As the
subgraph of $\Delta(G)$ on $\pi$ is complete, we deduce that $2\leq k\leq 3.$ If $k=2,$ then
$\rho(C/N)=\pi\cup \{p\}$ and $\pi(M/N)=\pi\cup \{q\},$ where $p$ and $q$ are distinct primes and
different from $r_i,1\leq i\leq k.$ Now, in the subgraph $\Delta(C/N\times M/N)$ of $\Delta(G),$ we
see that $p$ is adjacent to every prime in $\pi(M/N)$ and $q$ is adjacent to every prime in
$\rho(C/N),$ so $\Delta(C/N\times M/N)$ is isomorphic to $K_4,$ a contradiction. Similarly, if
$k=3,$ then the subgraph of $\Delta(G)$ on the set $\{r_i\}_{i=1}^3$ is a triangle and the vertex
$p\in(\pi(M/N)\cup \rho(C/N))-\pi$ is adjacent to every prime in $\pi,$ hence the subgraph of
$\Delta(G)$ on $\{p,r_1,r_2,r_3\}$ is a complete square, a contradiction.

Assume next that $|\rho(C/N)\cup \pi(M/N)|= 3.$ It follows that $|\pi(M/N)|=3$ and thus by applying
\cite[Table~1]{Huppert-Lempken}, we deduce that $\pi(G/C)=\pi(MC/C)=\pi(M/N),$ where $G/C$ is an
almost simple group with socle $M/N.$ As in the previous case, we have that $|\rho(C/N)|\geq 3$ and
thus $$\rho(C/N)=\pi(M/N)=\pi(G/N)=\{p_i\}_{i=1}^3.$$ Clearly, the subgraph of $\Delta(G)$ on
$\{p_i\}_{i=1}^3$ is a triangle. Since $|\rho(G)|\geq 5$ and $|\pi(G/N)|=3,$ we deduce that
$\rho(N)$ contains at least two distinct primes $r_i,i=1,2,$ such that
$r_i\not\in\rho(G/C)=\pi(M/N)$ for $i=1,2.$ Let $\theta_i\in \Irr(N)$ with $r_i\mid \theta_i(1)$
for $i=1,2.$ By \cite[Lemma~$4.2$]{Hung}, for each $i=1,2,$ either $\theta_i$ extends to $M$ or
$\psi_i(1)/\theta_i(1)$ is divisible by two distinct primes in $\pi(M/N)$ for some
$\psi_i\in\Irr(M|\theta_i).$ If $\theta_j$ is extendible to $M$ for some $j=1,2,$ then $r_j$ is
adjacent to every prime $p_i,1\leq i\leq 3,$ and so the subgraph of $\Delta(G)$ on
$\{r_j,p_1,p_2,p_3\}$ is a complete square, a contradiction. Hence, for each $j=1,2,$ $r_j$ is
adjacent to two distinct primes in $\{p_i\}_{i=1}^3.$ It follows that some vertex $p_m,$ $1\leq
m\leq 3,$ has degree at least four in $\Delta(G),$ which is impossible.

Therefore, $C=N$ and thus $G/N$ is almost simple with socle $M/N$ as wanted.
\end{proof}

We eliminate the graphs in Figures~\ref{Fig2}-\ref{Fig4} from being the prime graphs of any
nonsolvable groups in the next lemma.
\begin{lemma}\label{nontrivial tau}
Suppose that the prime graph $\Delta(G)$ of a nonsolvable group $G$ is a connected $3$-regular
graph with $|\rho(G)|\geq 6.$ Then $|\rho(G)|=6$ and $\Delta(G)$ is isomorphic to the graph in
Fig.~\ref{Fig1}.
\end{lemma}

\begin{proof}
By Proposition~\ref{classification of 3-regular graphs}, $\Delta(G)$ is  isomorphic to one of the
graphs in Figures \ref{Fig1}-\ref{Fig4} and $|\rho(G)|=6$ or $8.$ As the graph in Fig.~\ref{Fig1}
is the only graph of order $6,$ it suffices to show that $|\rho(G)|=6.$ By way of contradiction,
assume that $|\rho(G)|>6.$ Then $\rho(G)=8$ and so $\Delta(G)$ is one of the graphs in
Figures~\ref{Fig2}-\ref{Fig4}. As $\Delta(G)$ is $K_4$-free with maximal degree $d\leq 3.$ By
Lemma~\ref{solvable radical}, $G/N$ is almost simple with nonabelian simple socle $M/N,$ hence
$\rho(G/N)=\pi(G/N).$ Let $\tau=\rho(G)-\pi(G/N).$ Then $\tau\subseteq \rho(N).$ By
Lemma~\ref{almost simple groups}, $|\pi(G/N)|\leq 7,$ so $|\tau|=|\rho(G)|-|\pi(G/N)|\geq 1.$

\smallskip {\bf Claim~1.} $1\leq |\tau|\leq 2$ and the subgraph of $\Delta(G)$ on $\tau$ has no
edge.

If $|\tau|\geq 3,$ then since $\tau\subseteq \rho(N)$ with $N$ being solvable, there is an edge
among vertices in $\tau$ by \Pal. Thus it suffices to show that there is no edge among vertices in
$\tau.$ By way of contradiction, assume that $r\neq s\in\tau$ are joined to each other in
$\Delta(G)$ via $\psi\in\Irr(G).$ Let $\theta\in\Irr(N)$ be an irreducible constituent of $\psi_N.$
Since $rs\mid \psi(1)$ and $\psi(1)/\theta(1)\mid |G:N|,$ where $\gcd(rs,|G:N|)=1,$ we deduce that
$rs\mid \theta(1).$ By \cite[Lemma~4.2]{Hung}, either $\chi(1)/\theta(1)$ is divisible by two
distinct primes in $\pi(M/N)$ for some $\chi\in\Irr(M|\theta)$ or $\theta$ extends to
$\theta_0\in\Irr(M).$ The first case clearly cannot occur since otherwise $\chi(1)$ would be
divisible by at least $4$ distinct primes so $\Delta(M)$ would contain a complete square, a
contradiction. Thus the latter case holds. By Gallagher's Theorem \cite[Corollary~6.17]{Isaacs},
$\theta\lambda\in\Irr(M)$ for all $\lambda\in\Irr(M/N),$ hence $r$ is adjacent to every prime in
$\pi(M/N),$ where $|\pi(M/N)|\geq 3.$ Since $r$ is also adjacent to $s\not\in\pi(M/N),$ its degree
in $\Delta(G)$ is at least $4,$ a contradiction.

\smallskip
{\bf Claim~2.} $G=M$ and $G/N\cong\PSL_2(2^f)$ with $f\geq 10$ and $2\leq |\pi(2^f\pm 1)|\leq 3.$

Since $|\tau|\leq 2$ by Claim~1,  we have $|\pi(G/N)|=|\rho(G)|-|\tau|\geq 6.$ Lemma~\ref{almost
simple groups} now yields that $G/N=M/N\cong \PSL_2(2^f)$ with $f\geq 10$ and $|\pi(2^f\pm 1)|\leq
3.$ It follows that $|\pi(2^f\pm 1)|\geq 2$ as  $|\pi(2^{2f}-1)|=|\pi(G/N)-\{2\}|\geq 5.$

\smallskip
{\bf Claim~3.} For each $r\in \tau,$ there exist three distinct vertices $p_k\in\pi(G/N),k=1,2,3,$
such that the subgraphs of $\Delta(G)$ on $\{r,p_i,p_u\}$ and $\{r,p_i,p_v\}$ are triangles, where
$\{p_i,p_u,p_v\}=\{p_k\}_{k=1}^3.$ Hence, $\Delta(G)$ is isomorphic to the graph in
Fig.~\ref{Fig2}.

Let $r\in\tau.$ Then $r\mid \mu(1)$ for some $\mu\in\Irr(N).$ By \cite[Lemma~4.2]{Hung}, $r$
together with two distinct vertices in $\pi(G/N),$ say $p_1,p_2,$ will form a triangle in
$\Delta(G).$ Since $r$ is not adjacent to any primes in $\tau-\{r\}$ by Claim~1, $r$ is adjacent to
some prime in $\pi(G/N)-\{p_1,p_2\},$ say $p_3,$ via $\chi\in\Irr(G).$

Suppose that $p_3$ is not adjacent to $p_1$ nor $p_2.$ Since $r$ is of degree $3$ in $\Delta(G)$
and it is adjacent to every vertex $p_i,1\leq i\leq 3,$ it is not joined to any other vertices. In
particular, $\{r,p_3\}$ is not an edge of any triangle in $\Delta(G).$ Let $\theta\in\Irr(N)$ be an
irreducible constituent of $\chi_N.$ Then $\chi(1)/\theta(1)$ divides $|G/N|$ and thus
$\chi(1)/\theta(1)$ is prime to $r,$ therefore, $r\mid \theta(1).$ As $\{r,p_3\}\subseteq
\pi(\chi(1)),$ by our assumption on $r$ and $p_3$ we deduce that $\pi(\chi(1))=\{r,p_3\}.$ As
$|\pi(G/N)|\geq 6,$ $\theta$ is not extendible to $G$ as otherwise, $r$ would be adjacent to every
prime in $\pi(G/N)$ and so its degree in $\Delta(G)$ is at least six, a contradiction. It follows
that $\chi(1)/\theta(1)=p_3^a$ for some integer $a\geq 1.$ Also, as the Schur multiplier of
$G/N\cong\PSL_2(2^f)$ with $f\geq 10$ is trivial, $\theta$ is not $G$-invariant. Therefore,
$I=I_G(\theta)\lneq G.$ By Clifford's Theory, there exists $\phi\in\Irr(I|\theta)$ such that
$\chi=\phi^G$ and $\phi_N=e\theta$ for some integer $e\geq 1.$ Then
$\chi(1)=|G:I|\phi(1)=|G:I|e\theta(1).$ Hence, $\chi(1)/\theta(1)=e|G:I|=p_3^a.$ In particular,
$|G:I|$ is a prime power. Also, $|G:I|$ is divisible by the index of some maximal subgroup of
$G/N\cong\PSL_2(2^f).$ By \cite[Hauptsatz~II$.8.27$]{Huppert}, the indices of maximal subgroups of
$\PSL_2(2^f)$ are
\begin{equation}\label{eqn} 2^{f-1}(2^f+1), 2^{f-1}(2^f-1), 2^f+1,\frac{2^f(2^{2f}-1)}{2^b(2^{2b}-1)},
\end{equation}
where $f/b=n\geq 2$ is a prime. However, since $|\pi(2^f\pm 1)|\geq 2,$ none of these indices is a
prime power. This contradiction shows that $p_3$ must be adjacent to $p_1$ or $p_2,$ which proves
the first part of the claim. As the graphs in Figures~\ref{Fig3} and \ref{Fig4} do not contain any
two triangles sharing a common edge, $\Delta(G)$ must be the graph in Fig.~\ref{Fig2}.

\smallskip
{\bf The final contradiction.}

By Claims~$2$ and $3,$ $\Delta(G)$ is isomorphic to the graph in Figure~$2,$  $|\rho(G)|=8$ and
$G/N\cong\PSL_2(2^f)$ with $f\geq 10.$ Let $r\in \tau.$ By Claim~3, we can assume that the
subgraphs on $\{r,p_2,p_3\}$ and $\{r,p_2,p_1\}$ are two triangles in $\Delta(G).$ It follows that
the subgraph of $\Delta(G)$ on $\rho(G)-\{r\}$ has exactly two triangles sharing a common edge.
Hence, $|\pi(G/N)|\leq 6$ as if $|\pi(G/N)|\geq 7,$ then $|\pi(G/N)|=7$ and $\Delta(G/N)$ has two
disjoint triangles, which is impossible as $\pi(G/N)\subseteq \rho(G)-\{r\}.$ As $|\tau|\leq 2$ by
Claim~1, we deduce that $|\pi(G/N)|=6$ and so $\tau=\{r,s\}$ with $r\neq s.$ By Claim~1, $r$ and
$s$ are not adjacent and by Claim~3, $s$ is a common vertex of two triangles. It follows that the
subgraph of $\Delta(G)$ on $\pi(G/N)$ has no triangle with $6$ vertices, hence $\Delta(G/N)$ has no
triangle with $6$ vertices, contradicting \cite[Theorem~A]{Hung}. The proof is now
complete.\end{proof}

We now complete the proof of Theorem~A for nonsolvable groups.
\begin{theorem}\label{3-regular nonsolvable graph} Let $G$ be a nonsolvable group. If $\Delta(G)$ is a
cubic graph, then $\Delta(G)$ is isomorphic to a complete graph of order four.
\end{theorem}

\begin{proof}
Suppose that $G$ is a nonsolvable group such that $\Delta(G)$ is $3$-regular. Then $|\rho(G)|\geq
4$ and by Lemma~\ref{hand shaking}, $|\rho(G)|$ is even. By Lemma~\ref{disconnected regular
graphs}, $\Delta(G)$ is connected. If $|\rho(G)|=4,$ then we are done. So, we assume that
$|\rho(G)|\geq 6.$ By Lemma~\ref{nontrivial tau}, $|\rho(G)|=6$ and $\Delta(G)$ is the graph in
Fig.~\ref{Fig1}. Let $N$ be the solvable radical of $G$ and let $M$ be a normal subgroup of $G$
such that $M/N$ is a chief factor of $G.$ Lemma~\ref{solvable radical} implies that $G/N$ is almost
simple with simple socle $M/N.$ Let $\tau=\rho(G)-\pi(G/N).$ Then $\tau\subseteq\rho(N).$

\smallskip
{\bf Claim~1.} The subgraph of $\Delta(G)$ on $\tau$ has no edge, $|\tau|\leq 2$ and
$|\pi(M/N)|\geq 4.$

If $|\tau|\geq 3,$ then there is an edge among two distinct primes in $\tau$ by \Pal; and if
$|\pi(M/N)|< 4,$ then $|\pi(M/N)|=|\pi(G/N)|=3,$  which implies that
$|\tau|=|\rho(G)|-|\pi(G/N)|=3.$ Thus we need to show that there is no edge among primes in $\tau.$
This can be proved using the same argument as in the proof of Claim~1 in Lemma~\ref{nontrivial
tau}.

{\bf Claim~2.} $M/N\cong\PSL_2(q),$ where $q=p^f\geq 11,$ $p$ is a prime and $f\geq 1.$

If $|\pi(G/N)|\geq 5,$ then the result follows from Lemma~\ref{almost simple groups}. Hence, we
assume that $|\pi(G/N)|\leq 4.$ Since $4\leq |\pi(M/N)|\leq |\pi(G/N)|\leq 4,$ we obtain that
$|\pi(G/N)|=|\pi(M/N)|=4,$ and so $|\tau|=2.$ Writing $\tau=\{r,s\}.$ As $r$ and $s$ are not joined
to each other by Claim~1, $\{r,s\}$ is not an edge of any triangle in $\Delta(G)$ and thus the
subgraph of $\Delta(G)$ on $\pi(G/N)$ has no triangle and so $\Delta(G/N)$ has no triangle. Now the
result follows from \cite[Lemma~3.2]{Hung}.

{\bf Claim~3.} $\rho(G)=\pi(G/N).$

It suffices to show that $\tau=\emptyset.$ By way of contradiction, assume that $\tau$ is nonempty
and let $r\in\tau.$ By applying \cite[Lemma~4.2]{Hung}, we deduce that $\{r,p_1,p_2\}$ is the
vertex set of a triangle in $\Delta(G),$ where $\{p_1,p_2\}\subseteq \pi(M/N).$ Since $r$ has
degree three in $\Delta(G)$ and $r$ is not adjacent to any prime in $\tau,$ it must be adjacent to
some prime $p_3\in\pi(G/N).$ Hence $rp_3\mid\chi(1)$ for some $\chi\in\Irr(G).$ It follows that
$\pi(\chi(1))=\{r,p_3\}$ since $\{r,p_3\}$ is not an edge of any triangle in $\Delta(G).$ Let
$\psi\in\Irr(M)$ be an irreducible constituent of $\chi_M$ and let $\theta\in\Irr(N)$ be an
irreducible constituent of $\psi_N.$ Then $\theta$ is an irreducible constituent of $\chi_N.$ Since
$r\nmid |G/N|,$ we deduce that $r\mid \theta(1).$ As $|\pi(M/N)|\geq 4,$ we deduce that $\theta$ is
not extendible to $M$ and thus $\psi(1)/\theta(1)=p_3^a$ for some integer $a\geq 1.$

Assume first that $\theta$ is not $M$-invariant. Then $I=I_M(\theta)\lneq M.$ Let
$\phi\in\Irr(I|\theta)$ such that $\phi^M=\psi.$ We have that $\phi_N=e\theta$ for some integer
$e\geq 1.$ Now $\psi(1)=\phi^M(1)=|M:I|e\theta(1),$ so $|M:I|e=p_3^a.$ In particular, $|M:I|$ is a
power of $p_3.$ Using  \cite[Hauptsatz~II$.8.27$]{Huppert}, we can deduce that either $q=11$ and
$I/N\cong \textrm{A}_5$ with $|M:I|=11$ or $|M:I|=q+1$ is a power of $p_3$ and $I/N$ is the
normalizer in $M/N$ of a Sylow $p$-subgroup of $M/N\cong\PSL_2(q),$ where $q$ is a power of a prime
$p.$ In both cases, we have that $\gcd(|M:I|,|I:N|)=1$ and $I/N$ is nonabelian. Thus if $e>1$ and
$u\mid e$ is a prime, then since $e\mid |I/N|,$ $e\neq p_3,$ which is impossible. Therefore, $e=1,$
which implies that $\phi$ is an extension of $\theta$ to $I.$ Since $I/N$ is nonabelian, we can
find $\lambda\in\Irr(I/N)$ with $\lambda(1)>1.$ By Gallagher's Theorem and Clifford's Theory, we
obtain that $(\phi\lambda)^M\in\Irr(M)$ and so
$(\phi\lambda)^M(1)=|M:I|\lambda(1)\theta(1)\in\cd(M|\theta).$ Since $\lambda(1)>1$ and
$\lambda(1)\mid |I/N|,$ we can find a prime divisor $v$ of $\lambda(1)$ such that $v\neq p_3$ and hence
$\{r,p_3,v\}$ is the vertex set of some triangle in $\Delta(G),$ which is impossible.

Therefore, we can assume that $\theta$ is $M$-invariant but not extendible to $M.$ Since $q\geq
11,$ the Schur representation group of $M/N$ is $\SL_2(q)$ and thus by the theory of character
triple isomorphisms in \cite[Chapter~11]{Isaacs}, we deduce that $\theta(1)(q\pm 1)\in
\cd(M|\theta).$ In particular, $r$  is adjacent to every prime in $\pi(q^2-1).$ It follows that
$|\pi(q^2-1)|=3$ since $|\pi(M/N)|\geq 4$ and the degree of $r$ in $\Delta(G)$ is three. However,
as $q\geq 11$ is odd, we see that the subgraph of $\Delta(G)$ on $\{r\}\cup\pi(q^2-1)$ has two
triangles sharing a common edge, which is impossible.

\smallskip
{\bf  The final contradiction.}

By the previous claim, we have that  $|\pi(G/N)|=|\rho(G)|=6.$ By Lemma~\ref{almost simple groups},
$G=M$ and $G/N\cong\PSL_2(2^f)$ with $f\geq 10.$ Writing $\pi(2^f+\delta)=\{p_k\}_{k=1}^3$ and
$\pi(2^f-\delta)=\{p_k\}_{k=4}^5,$ where $\delta=\pm 1.$

Since $\Delta(G)$ is connected but $\Delta(G/N)$ is not, we deduce that $N$ is nontrivial. Observe
that in the graph $\Delta(G),$ the prime $2$ is adjacent to exactly one prime in $\{p_k\}_{k=1}^3,$
say $p_i.$ Hence, there exists $\chi\in\Irr(G)$ with $2p_i\mid \chi(1).$ Let $\theta\in\Irr(N)$ be
an irreducible constituent of $\chi_N.$ Since $\{2\}$ is isolated in $\Delta(G/N),$ we deduce that
$\theta\neq 1_N.$ As $\Delta(G)$ has no triangle which contains an edge with vertex set
$\{2,p_i\},$ we must have that $\pi(\chi(1))=\{2,p_i\}.$

Assume first that $\theta$ is not $G$-invariant and put $I=I_G(\theta).$ Then $I/N$ is a proper
subgroup of the simple group $G/N\cong\PSL_2(2^f).$ It follows that $|G:I|$ is divisible by the
index of some maximal subgroup of $G/N.$ Hence, $|G:I|$ is divisible by one of the numbers in
\eqref{eqn}. Now the first three possibilities cannot occur since $\chi(1)$ and hence $|G:I|$ is
divisible by only one odd prime. For the last case, if $f=2b,$ then the odd part of last index is
$2^f+1$ which is also divisible by two distinct odd primes; therefore, we can assume that
$f/b=n\geq 3.$ In this case, $(2^{2f}-1)/(2^{2b}-1)$ must be a power of $p_i,$ which is impossible
by applying \cite[Lemma~2.4]{Hung} since $f\geq 10.$

Assume now that $\theta$ is $G$-invariant. Since the Schur multiplier of $G/N\cong \PSL_2(2^f)$
with $f\geq 10$ is trivial, we deduce from \cite[Theorem~11.7]{Isaacs} that $\theta$ is extendible
to $\theta_0\in\Irr(G).$ By applying Gallagher's Theorem \cite[Corollary~6.17]{Isaacs}, we have
that \[\chi(1)\in\{\theta_0(1)=\theta(1),2^f\theta(1),(2^f+ \delta)\theta(1),(2^f-
\delta)\theta(1)\}\subseteq\cd(G).\] As $\chi(1)$ is divisible by only one odd prime, we deduce
that $\chi(1)=\theta(1)$ or $2^f\theta(1).$ In both cases, we deduce that $p_i\mid \theta(1)$ and
thus $(2^f-\delta)\theta(1)$ is divisible by $p_4,p_5$ and $p_i,$ and so these three primes form a
triangle in $\Delta(G),$ which is impossible. The proof is now complete.
\end{proof}

\begin{proof}[\bf Proof of Theorem~A]
Let $G$ be a group with prime graph $\Delta(G).$ If $\Delta(G)$ is a complete square, then it is a
cubic graph. Conversely, if $\Delta(G)$ is a cubic graph, then it is a complete square by
Theorems~\ref{3-regular solvable graph} and \ref{3-regular nonsolvable graph}.
\end{proof}
 \section{Examples}
Let $H$ and $K$ be groups such that $\rho(H)\cap \rho(K)=\emptyset$ and both $\Delta(H)$ and
$\Delta(K)$ are disconnected graphs of order two. Let $G=H\times K.$ Then $\Delta(G)$ is a square
and $G$ is solvable. Conversely, every group whose prime graph is a square must be a direct product
of two groups $H$ and $K$ satisfying the properties above. (See \cite[Corollary~C]{Lewis-White13}).

\medskip
Based on this example, for each even integer $k\geq 4,$ we can construct a solvable group $G$ whose
prime graph is $k$-regular with $|\rho(G)|=k+2.$  Writing $k=4\ell+r$ with $r\in\{0,2\}$ and let
$n=k+2.$ Let $G_i,i=1,2,\cdots, \ell,$ be groups whose prime graphs are squares such that
$\rho(G_i)\cap \rho(G_j)=\emptyset,$ for all $1\leq i\neq j\leq \ell.$ Let $G_0$ be a group whose
prime graph $\Delta(G_0)$ is a square if $r=2$ and is a disconnected graph with two vertices if
$r=0,$ where $\rho(G_0)\cap (\cup_{i=1}^\ell \rho(G_i))=\emptyset.$ It follows that
$|\rho(G_0)|=r+2.$ Let $G=\prod_{i=0}^\ell G_i$ be the direct product of all $G_i's,$ for $0\leq
i\leq \ell.$ Then $|\rho(G)|=r+2+4\ell=k+2$ and  it is not hard to verify that $\Delta(G)$ is
$k$-regular and since each $G_i, 0\leq i\leq \ell$ is solvable, we deduce that $G$ is solvable.
Fig.~\ref{Fig5} gives an example of a $4$-regular graph of order six constructed in this way.

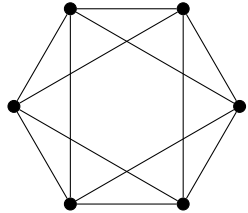
\begin{figure}[ht]
\begin{center}
\[\begin{tikzpicture}[x=1.5cm, y=1.5cm]
    \foreach \i in {1, 2, 3, 4, 5, 6} {
        \setcounter{Angle}{\i * 360 / 6}
        \vertex[fill] (p\i) at (\theAngle:1) [label=\theAngle:]{};
    }
    \path
        (p1) edge (p2)
        (p2) edge (p3)
        (p3) edge (p4)
        (p4) edge (p5)
        (p5) edge (p6)
        (p6) edge (p1)
        (p1) edge (p3)
        (p1) edge (p5)
        (p2) edge (p4)
        (p2) edge (p6)
        (p3) edge (p5)
        (p6) edge (p4)

    ;
\end{tikzpicture}\]

\end{center}
\caption{A quartic graph of order six}\label{Fig5}
\end{figure}

In view of Theorem~A, Proposition~\ref{regular graphs of small valency} and the examples above, we
formulate the following conjecture.
\begin{conjecture} Let $G$ be a group and let $k\geq 2$ be an integer. Suppose that $\Delta(G)$ is $k$-regular. Then
\begin{enumerate}
\item[$(1)$] If $k\geq 5$ is odd, then $\Delta(G)$ is a complete graph of order $k+1.$
\item[$(2)$] If $k\geq 4$ is even, then $\Delta(G)$ is either a complete graph of order $k+1$ or a $k$-regular graph of order $k+2.$
\item[$(3)$] If $|\rho(G)|=k+2,$ then $G$ is solvable.
\end{enumerate}
\end{conjecture}

\section*{Acknowledgment} The author is grateful to the referee for his or her corrections and
suggestions and to Michael Henning for his help during the preparation of this work.

\end{document}